\documentclass{amsart}
\usepackage{amssymb}
\usepackage{amsmath}
\usepackage{amsfonts}
\usepackage{hyperref}
\usepackage{yhmath}
\usepackage{framed}
\newtheorem{theorem}{Theorem}[section]
\newtheorem{lemma}[theorem]{Lemma}

\renewcommand{\l}{\lambda}
\newcommand{\R}{\mathbb{R}}

\newcommand{\p}{\partial}

\begin{document}
\title[Almost optimal local well-posedness for imBq equations]{Almost optimal local well-posedness for improved modified Boussinesq equations}

\author{Dan-Andrei Geba and Bai Lin}

\address{Department of Mathematics, University of Rochester, Rochester, NY 14627}
\email{dangeba@math.rochester.edu}
\address{Department of Mathematics, University of Rochester, Rochester, NY 14627}
\email{blin13@ur.rochester.edu}
\date{}

\begin{abstract}
In this article, we investigate a class of improved modified Boussinesq equations, for which we provide first an alternate proof of local well-posedness in the space $(H^s\cap L^\infty)\times (H^s\cap L^\infty)(\R)$ ($s\geq 0$) to the one obtained by Constantin and Molinet \cite{CM02}. Secondly, we show that the associated flow map is not smooth when considered from $H^s\times H^s(\R)$ into $H^s(\R)$ for $s<0$, thus providing a threshold for the regularity needed to perform a Picard iteration for these equations. 
\end{abstract}

\subjclass[2000]{35B30, 35Q55}
\keywords{improved modified Boussinesq equation, well-posedness, ill-posedness.}

\maketitle

\section{Introduction}

\subsection{Background}
Our goal is to study the initial value problem (IVP)
\begin{equation}
\begin{cases}
u_{tt}-u_{xx}-u_{xxtt}\,=\,(f(u))_{xx}, \qquad u=u(t,x)\in \mathbb{R}_+\times \R \to \R,\\
u(0,x)\,=\,u_0(x),\qquad u_t(0,x)\,=\,u_1(x),\\
\end{cases}
\label{main}
\end{equation}
for which the differential equation is known in the literature as the improved modified Boussinesq (imBq) equation. Initially, Makhankov \cite{M78} derived the equation with $f(u)=u^2$ in the context of ion-sound wave propagation and mentioned the one with $f(u)=u^3$ as modeling nonlinear Alfv\'{e}n waves. Later, Clarkson, LeVeque, and Saxton \cite{CLS86} discovered that the equations with either $f(u)=u^3/3$ or $f(u)=u^5/5$ describe the propagation of longitudinal deformation waves in an elastic rod. 

The imBq equation is also known as an improved frequency dispersion version of the classical Boussinesq equation
\begin{equation*}
u_{tt}-u_{xx}-u_{xxxx}\,=\,(u^2)_{xx},
\end{equation*}
derived in relation to shallow water waves. The latter has the dispersive relation 
\[
\omega^2=k^2-k^4, 
\]
which leads to a nonphysical instability when $k>1$. This is not the case for the imBq equation, whose 
dispersive relation is given by 
\[
\omega^2=\frac{k^2}{1+k^2}.
\]
In the same context, another well-known improved frequency dispersion version of the classical Boussinesq equation is the ``good" or ``well-posed" Boussinesq equation
\begin{equation}
u_{tt}-u_{xx}+u_{xxxx}\,=\,(f(u))_{xx},
\label{gBsq}
\end{equation}
which was found to describe electromagnetic waves in nonlinear dielectrics, magnetoelastic waves in antiferromagnets, and shape-memory alloys.

Past investigations concerning the IVP \eqref{main} mainly focused on two directions. The first one concentrated on the local existence and uniqueness of various types of solution (e.g., strong, classical) as well as on sufficient conditions for the global existence or the blow-up in finite time of such solutions. We mention here work by Constantin and Molinet \cite{CM02}, who looked at the equation
\begin{equation}
u_{tt}-u_{xxtt}\,=\,(F(u))_{xx}, \qquad F\in C^\infty (\R), \ F(0)=0, \label{CM-eq}
\end{equation}
and showed that the associated IVP is locally well-posed (LWP) for $(u(0),u_t(0))\in (H^s\cap L^\infty)(\R)\times (H^s\cap L^\infty)(\R)$, with $s\geq 0$ being arbitrary. Moreover, the same paper contains both continuation criterions for local-in-time solutions to be extended into global ones and conditions on $F$ which guarantee either global solutions or blow-up in finite time for certain data profiles. Similar results were obtained by Wang and Chen \cite{WC02} for the multidimensional problem (i.e, $x\in\R^n$, $n\geq 2$, and every $\p^2_x$ is replaced by $\Delta$).

The other type of question that was studied in connection to the IVP \eqref{main} is the existence and scattering of global small amplitude solutions. A very formal description of this question is as follows: what are the values of $p>1$ for which global, small $H^s$ solutions to \eqref{main}, with $|f(u)|\simeq |u|^{p}$, scatter?  Cho and Ozawa \cite{CO07} gave an almost optimal answer to this question both for \eqref{main} and the IVP for the ``good" Boussinesq equation \eqref{gBsq}. We refer the interested reader to this article and references therein for a comprehensive discussion of this issue.

\subsection{Description and statement of main results}
One topic which is usually studied in relation to evolution equations, especially dispersive ones, is the ill-posedness (IP) of the associated IVP. To our best knowledge, such an inquiry has not been conducted yet for \eqref{main}. The goal of this article is to do just that, in the case when $f(u)=\pm\, u^p$ and $p>1$ is an integer. Our results are in the same spirit with the ones originally obtained by Bourgain \cite{B97} and Tzvetkov \cite{T99} for the KdV equation and then also derived for other dispersive equations (e.g., Molinet and Ribaud \cite{MR04}, Bona and Tzvetkov \cite{BT09}, Geba, Himonas, and Karapetyan \cite{GHK14}). 

They establish loss of smoothness for the flow map, which is defined for a fixed time $t$ as 
\[
(u_0,u_1)\,\mapsto\,S(t)(u_0,u_1):=u(t).
\]
The loss of regularity occurs when the domain of the flow map is chosen to be $H^s\times H^s(\R)$, with $s<0$ being arbitrary. To argue for the optimality claimed in the title, we show that \eqref{main} is LWP in $(H^s\cap L^\infty)\times (H^s\cap L^\infty)(\R)$ when $s\geq 0$, by running a contraction argument for one of its integral formulations. In particular, this implies that the flow map is smooth as a map from $(H^s\cap L^\infty)\times (H^s\cap L^\infty)(\R)$ to $(H^s\cap L^\infty)(\R)$ for all times in the interval of existence. It is not clear that this conclusion can be drawn, at least easily, from the analysis done by Constantin and Molinet in \cite{CM02}. There, the equation \eqref{CM-eq} is recast in the form of an ODE system in Banach spaces and the LWP is obtained by classical Picard iteration. Another reason for the inclusion of our LWP argument is that it contains some new harmonic analysis facts that may be of independent interest. 

Following this, we derive the integral version of the IVP \eqref{main}, which is the central object of study from this point onward. This is obtained by first rewriting the imBq equation as
\begin{equation}
u_{tt}+P(D)u=-P(D)f(u), \qquad P(D):=\mathcal{F}_\xi^{-1}\frac{\xi^2}{1+\xi^2}\mathcal{F}_x,
\label{eq-PD}
\end{equation}
and then applying Duhamel's principle to infer
\begin{equation}
u(t)=L(u_0,u_1)(t)- \int_0^t\,L \left( 0, P(D)(f(u(\tau)))\right)(t-\tau)\,d\tau,
\label{main-int}
\end{equation}
where
\begin{equation}
\aligned
\widehat{L(v_0,v_1)(t)}(&\xi):=\cos(t \lambda(\xi))\, \widehat{v}_0(\xi)+\frac{\sin(t \lambda(\xi))}{\lambda(\xi)} \,\widehat{v}_1(\xi),\\ 
&\lambda(\xi):=\frac{|\xi|}{\langle\xi\rangle}=\frac{|\xi|}{(1+\xi^2)^{1/2}}.
\endaligned
\label{L}
\end{equation}

We can now state our main results.

\begin{theorem}
Consider the integral equation \eqref{main-int} with  $f(u)=\pm\, u^p$ and $p>1$ being an arbitrary integer. 

i) (LWP) If $s\geq 0$ and $(u_0,u_1)\in(H^s\cap L^\infty)\times (H^s\cap L^\infty)(\R)$, then there exist 
\[
T=T(\|(u_0,u_1)\|_{(H^s\cap L^\infty)\times (H^s\cap L^\infty)(\R)})>0
\] 
and a unique solution $u$ satisfying
\[
u\in C([0,T], (H^s\cap L^\infty)(\R)).
\] 
Moreover, 
\[S(t): (H^s\cap L^\infty)\times (H^s\cap L^\infty)(\R)\to (H^s\cap L^\infty)(\R),\quad S(t)(u_0,u_1):=u(t),\]
is smooth for all $t\in [0,T]$.

ii) (IP) If $s<0$, then there exists $T>0$ such that 
\[
S(t):  H^{s} \times H^{s}(\R) \to H^{s}(\R)
\]
does not admit a $p$-th order Fr\'{e}chet derivative at zero for all $0<t<T$.
\label{mainth}
\end{theorem}
\noindent The LWP part of this theorem will be addressed in the next section, whereas the argument for IP will occupy the final one.


\section{LWP argument}

In proving the LWP claim, we rely on the classical approach of verifying that the right-hand side of \eqref{main-int}, when seen as a functional in $u$ (with the data $u_0$ and $u_1$ being fixed), is a contraction on a suitably chosen closed ball of a Banach space. 

For this purpose, we are first concerned with the mapping properties of the multiplier operators $P(D)$ (defined in \eqref{eq-PD}) and
\begin{equation}
Q_t(D):=\mathcal{F}_\xi^{-1}\cos(t\l(\xi))\mathcal{F}_x, \qquad R_t(D):=\mathcal{F}_\xi^{-1}\frac{\sin(t\l(\xi))}{\l(\xi)}\mathcal{F}_x,
\label{qtrt-def}
\end{equation}
where $t\in\R$ is arbitrary, yet fixed. Given the trivial bounds
\begin{equation*}
0\leq \l^2(\xi)=\frac{\xi^2}{1+\xi^2}<1, \qquad |\cos(t\l(\xi))|\leq 1, \qquad \left|\frac{\sin(t\l(\xi))}{\l(\xi)}\right|\leq |t|,
\end{equation*}
Plancherel's formula implies\footnote{From here on out, for a functional space $Y$, we write $Y=Y(\R)$ as the majority of such norms refers to this particular situation.}
\begin{equation}
\|P(D)v\|_{H^s}\leq \|v\|_{H^s}, \quad \|Q_t(D)v\|_{H^s}\leq \|v\|_{H^s}, \quad \|R_t(D)v\|_{H^s}\leq |t| \|v\|_{H^s}.
\label{hs-bd}
\end{equation}
 
Next, and this is one of the novelties in our paper, we show that the symbols of these operators are also Fourier multipliers on $L^\infty$ in the sense of Definition 6.1.1 in Bergh-L\"{o}fstr\"{o}m \cite{BL76}. By comparison, Constantin and Molinet proved in \cite{CM02} that $P(D)$ maps $H^s\cap L^\infty$ into itself for all $s\geq 0$. In arguing for this claim, we rely on a number of facts, some of which are contained in the book by Bergh and L\"{o}fstr\"{o}m. One\footnote{This is the conclusion of Exercise 16 on page 164 in \cite{BL76}.} is that the homogeneous Besov space $\dot{B}^{n/2}_{2,1}(\R^n)$ is a subspace of the normed space of Fourier multipliers on $L^\infty(\R^n)$. A second fact\footnote{This can be inferred from Theorem 6.3.1 in \cite{BL76}. See also Section 3.2 in Shatah-Struwe \cite{SS98}.} is the equivalence between the original seminorm for $\dot{B}^{n/2}_{2,1}(\R^n)$ and the one given by
\begin{equation*}
\|w\|^*_{\dot{B}^{n/2}_{2,1}(\R^n)}:=\int_{\R^n}\frac{\|w(\cdot+h)-w(\cdot)\|_{L^2(\R^n)}}{|h|^{n+\frac{1}{2}}}\,dh.
\end{equation*}
We will also use the following integration result\footnote{This is a special case of Lemma 4.2 in Ginibre-Tsutsumi-Velo \cite{GTV97}.}:
\begin{equation}
\int_\R\frac{1}{\langle z-a\rangle^2\langle z-b\rangle^4}\,dz\lesssim \frac{1}{\langle a-b\rangle^2}, \qquad (\forall)\,a, b\in\R.
\label{int}
\end{equation}

\begin{lemma}
The symbols $m_1(\xi)=\l^2(\xi)$, $m_2(\xi)=e^{\pm it\l(\xi)}$, and $m_3(\xi)=\sin(t\l(\xi))/\l(\xi)$ are all Fourier multipliers on $L^\infty$ and
\begin{align}
&\|P(D)v\|_{L^\infty}\lesssim \|v\|_{L^\infty}, \label{p-li}\\ 
\|Q_t(D)v\|_{L^\infty}\lesssim |t|&\|v\|_{L^\infty},\qquad
\|R_t(D)v\|_{L^\infty}\lesssim \max\{|t|, |t|^3\} \|v\|_{L^\infty}.\label{qrt-li}
\end{align}
\end{lemma}

\begin{proof}
Based on the facts listed above, it is clear that the lemma is proved if we show that
\begin{equation*}
\|m_1\|^*_{\dot{B}^{1/2}_{2,1}}\lesssim 1,\qquad \|m_2\|^*_{\dot{B}^{1/2}_{2,1}}\lesssim |t|, \qquad \|m_3\|^*_{\dot{B}^{1/2}_{2,1}}\lesssim \max\{|t|, |t|^3\}.
\end{equation*}
A direct application of the Cauchy-Schwarz inequality yields
\begin{equation*}
\|w(\cdot+h)-w(\cdot)\|_{L^2}\leq |h|\|w'\|_{L^2}
\end{equation*}
and straightforward computations provide us with the bounds
\begin{equation*}
|m'_1(\xi)|\lesssim \frac{1}{\langle\xi\rangle^3},\qquad |m'_2(\xi)|\lesssim \frac{|t|}{\langle\xi\rangle^3}, \qquad |m'_3(\xi)|\lesssim \frac{|t|^3}{\langle\xi\rangle^3}.
\end{equation*}
Therefore, we infer that
\begin{align*}
\int_{|h|\leq 2}\frac{\|m_1(\cdot+h)-m_1(\cdot)\|_{L^2}}{|h|^{3/2}}\,dh&\lesssim 1,\\
\int_{|h|\leq 2}\frac{\|m_2(\cdot+h)-m_2(\cdot)\|_{L^2}}{|h|^{3/2}}\,dh&\lesssim|t|,\\
\int_{|h|\leq 2}\frac{\|m_3(\cdot+h)-m_3(\cdot)\|_{L^2}}{|h|^{3/2}}\,dh&\lesssim|t|^3.
\end{align*}

All which is left to discuss is the scenario when $|h|\geq 2$. In this case, since
\[
|\xi|+|\xi+h|\geq |h|,
\]
it follows that 
\begin{equation}
\max\{|\xi|, |\xi+h|\}\simeq\max\{\langle\xi\rangle, \langle\xi+h\rangle\}.
\label{hg2}
\end{equation}
Coupled to
\begin{equation}
\l(\xi+h)-\l(\xi)=\frac{h(2\xi+h)}{\langle\xi+h\rangle\langle\xi\rangle(|\xi+h|\langle\xi\rangle+|\xi|\langle\xi+h\rangle)},
\label{la-lb}
\end{equation}
this implies 
\begin{equation*}
|\l(\xi+h)-\l(\xi)|\lesssim |h|\max\left\{\frac{1}{\langle\xi+h\rangle\langle\xi\rangle^2}, \frac{1}{\langle\xi+h\rangle^2\langle\xi\rangle}\right\}.
\end{equation*}
As a consequence of \eqref{int}, we deduce
\begin{equation}
\int_{|h|\geq 2}\frac{\|\l(\cdot+h)-\l(\cdot)\|_{L^2}}{|h|^{3/2}}\,dh\lesssim 1.
\label{lhg2}
\end{equation}

On the other hand, due to $0\leq\l(\xi)<1$, we have
\begin{equation}
|m_1(\xi+h)-m_1(\xi)|\leq 2|\l(\xi+h)-\l(\xi)|.
\label{m1l}
\end{equation}
We also have
\begin{equation}
|m_2(\xi+h)-m_2(\xi)|= 2\left|\sin(t(\l(\xi+h)-\l(\xi))/2)\right|\leq |t||\l(\xi+h)-\l(\xi)|.
\label{m2l}
\end{equation}
For $m_3$, we can write
\[
m_3(\xi+h)-m_3(\xi)= \frac{\sin(t\l(\xi+h))-\sin(t\l(\xi))}{\l(\xi+h)}+\sin(t\l(\xi))\left(\frac{1}{\l(\xi+h)}-\frac{1}{\l(\xi)}\right)
\]  
which leads to
\begin{equation*}
|m_3(\xi+h)-m_3(\xi)|\leq \frac{2|t||\l(\xi+h)-\l(\xi)|}{\l(\xi+h)}.
\end{equation*}
By symmetry, we obtain
\begin{equation*}
|m_3(\xi+h)-m_3(\xi)|\leq \frac{2|t||\l(\xi+h)-\l(\xi)|}{\max\{\l(\xi+h),\l(\xi)\}}.
\end{equation*}
Due to \eqref{hg2}, we infer
\begin{equation*}
\max\{\l(\xi+h),\l(\xi)\}\simeq 1,
\end{equation*}
and, subsequently,
\begin{equation*}
|m_3(\xi+h)-m_3(\xi)|\lesssim |t||\l(\xi+h)-\l(\xi)|.
\end{equation*}
Jointly with \eqref{lhg2}-\eqref{m2l}, this estimate implies
\begin{align*}
\int_{|h|\geq 2}\frac{\|m_1(\cdot+h)-m_1(\cdot)\|_{L^2}}{|h|^{3/2}}\,dh&\lesssim 1,\\
\int_{|h|\geq 2}\frac{\|m_2(\cdot+h)-m_2(\cdot)\|_{L^2}}{|h|^{3/2}}\,dh&\lesssim|t|,\\
\int_{|h|\geq 2}\frac{\|m_3(\cdot+h)-m_3(\cdot)\|_{L^2}}{|h|^{3/2}}\,dh&\lesssim|t|,
\end{align*}
and the proof of the lemma is concluded.
\end{proof}

Now, we can start in earnest the LWP argument. For fixed $u_0$ and $u_1$, we denote the right-hand side of \eqref{main-int} by $Z_{u_0,u_1}(u)$ and, using \eqref{eq-PD} and \eqref{qtrt-def}, we infer
\begin{equation*}
Z_{u_0,u_1}(u)= Q_t(D)u_0+R_t(D)u_1-\int_0^tR_{t-\tau}(D)(P(D)(f(u(\tau))))\,d\tau.
\end{equation*}
The well-known Moser-type estimate
\begin{equation*}
\|vw\|_{H^s}\lesssim \|v\|_{H^s}\|w\|_{L^\infty}+\|v\|_{L^\infty}\|w\|_{H^s},
\end{equation*}
which is valid for all $s\geq 0$, implies that $H^s\cap L^\infty$ is an algebra in this case. If we assume $0\leq t\leq T$ and take advantage of \eqref{hs-bd},\eqref{p-li}, and \eqref{qrt-li}, we deduce
\begin{equation*}
\aligned
\|Z_{u_0,u_1}(u)(t)\|_{H^s\cap L^\infty}\lesssim\,&\max\{1,t\} \|u_0\|_{H^s\cap L^\infty} +\max\{t,t^3\} \|u_1\|_{H^s\cap L^\infty}\\
&+\int_0^t\max\{t-\tau,(t-\tau)^3\} \|f(u(\tau))\|_{H^s\cap L^\infty}\,d\tau.
\endaligned
\end{equation*}
Choosing now $f(u)$ to be as in Theorem \ref{mainth}, it follows that
\begin{equation*}
\aligned
\|Z_{u_0,u_1}(u)\|_{C([0,T]; H^s\cap L^\infty)}\lesssim\,&\max\{1,T\} \|u_0\|_{H^s\cap L^\infty} +\max\{T,T^3\} \|u_1\|_{H^s\cap L^\infty}\\
&+\max\{T^2,T^4\} \|u\|^p_{C([0,T]; H^s\cap L^\infty)}.
\endaligned
\end{equation*}
Therefore, by working with $T<1$, we obtain that for\footnote{This is the ball of radius $R$ centered at the origin in the Banach space $C([0,T]; H^s\cap L^\infty)$.} $B(0,R)\subset C([0,T]; H^s\cap L^\infty)$ one has  
\begin{equation*}
u\in B(0,R)\mapsto Z_{u_0,u_1}(u)\in B(0,R)
\end{equation*}
if
\begin{equation*}
R\simeq \|u_0\|_{H^s\cap L^\infty} + \|u_1\|_{H^s\cap L^\infty}\quad\text{and}\quad T\lesssim R^{-\frac{p-1}{2}}.
\end{equation*}
Furthermore, using a similar argument, we derive
\begin{equation*}
\aligned
\|Z_{u_0,u_1}&(u)-Z_{u_0,u_1}(\tilde{u})\|_{C([0,T]; H^s\cap L^\infty)}\\
&\lesssim\,T^2 \|u-\tilde{u}\|_{C([0,T]; H^s\cap L^\infty)}\left(\|u\|^{p-1}_{C([0,T]; H^s\cap L^\infty)}+\|\tilde{u}\|^{p-1}_{C([0,T]; H^s\cap L^\infty)}\right)\\
&\lesssim\,T^2 R^{p-1}\|u-\tilde{u}\|_{C([0,T]; H^s\cap L^\infty)}.
\endaligned
\end{equation*}
Thus, with an eventual additional adjustment on the size of $T$, we conclude that $u\mapsto Z_{u_0,u_1}(u)$ is a contraction on the ball $B(0,R)$ and, consequently, a unique solution to the integral equation \eqref{main-int} exists on the time interval $[0,T]$. The smoothness of the flow map follows then by an application of the analytic version of the implicit function theorem. With this, the LWP part of Theorem \ref{mainth} has been proved.

\section{IP argument}
As explained in the introduction, the scheme for proving IP consists in showing that a putative flow map, when acting from $H^s\times H^s$ into $H^s$, fails to be smooth for any $s<0$. Our approach is very similar to the one used in \cite{GHK14}, from which it borrows the main framework. 

From \eqref{main-int}, we see that if $f(u)=\pm u^p$ then 
\begin{equation*}
\aligned
S(t)(u_0,u_1)\,=\,&L(u_0,u_1)(t)\\
&\mp\, \int_0^t\,L \left( 0, P(D)\left((S(\tau)(u_0,u_1))^p\right)\right)(t-\tau)\,d\tau, \quad (\forall)\, t\in[0,T],
\endaligned
\end{equation*}
where $T>0$ is such that the flow map makes sense on $[0,T]$ near the origin in $H^s\times H^s$. When the flow map is sufficiently regular, we can use this equation to compute explicitly (by relying on implicit differentiation) its Fr\'{e}chet derivatives at the origin. Precisely, we have 
\begin{equation*}
\aligned
DS(t)_{(v_0,v_1)}&(u_0,u_1)=L(v_0,v_1)(t)\\
&\mp p\int_0^t\,L \left( 0,P(D)\left(DS(\tau)_{(v_0,v_1)}(u_0,u_1)\left(S(\tau)(u_0,u_1)\right)^{p-1}\right)\right)(t-\tau)\,d\tau,
\endaligned\label{ds}
\end{equation*}
where $DS(t)_{(v_0,v_1)}(u_0,u_1)$ stands for the first order Fr\'{e}chet derivative of the flow map at $(u_0,u_1)$, evaluated for $(v_0,v_1)$. Given that LWP ensures $S(t)(0,0)=0$, we deduce
\begin{equation*}
DS(t)_{(v_0,v_1)}(0,0)\,=\,L(v_0,v_1)(t).
\label{ds0}
\end{equation*} 
Arguing along the same lines, we further derive
\begin{equation*}
D^kS(t)_{(v^1_0,v^1_1),\ldots, (v^k_0,v^k_1)}(0,0)=0, \qquad (\forall)\,1<k<p,
\end{equation*}
and, eventually,
\begin{equation}
\aligned
 D^p&S(t)_{(v^1_0,v^1_1),\ldots, (v^p_0,v^p_1)}(0,0)\\
 &\qquad=\mp p!\,\int_0^t\,L \left( 0,P(D)\left(L(v^1_0,v^1_1)(\tau)\,\cdot \ldots\cdot\,L(v^p_0,v^p_1)(\tau)\right)\right)(t-\tau)\,d\tau,
\endaligned
\label{dps}
\end{equation}

Hence, if the flow map had $C^p$ regularity at the origin, the estimate  
\begin{equation*}
 \left\|D^pS(t)_{(v^1_0,v^1_1),\ldots, (v^p_0,v^p_1)}(0,0)\right\|_{H^s}\,\lesssim\,\prod_{j=1}^{p}\,\left\|(v^j_0, v^j_1)\right\|_{H^{s}\times H^s}
\end{equation*}
would hold true uniformly for $t\in [0,T]$. However, when $s<0$, we show that this bound fails by constructing a sequence $\left(u^N_0,u^N_1\right)_N\subset H^s\times H^{s}$ satisfying
\begin{equation}
\lim_{N\to \infty}\, \frac{\left\|D^pS(t)_{(u^N_0,u^N_1),\ldots, (u^N_0,u^N_1)}(0,0)\right\|_{H^s}}{\left(\left\|u_0^N\right\|_{H^{s}} + \left\|u_1^N\right\|_ {H^{{s}}}\right)^p}\,=\,\infty, \qquad (\forall)\, 0 < t < T.
\label{ratio}
\end{equation}

For ease of notation, we use onward the abbreviation
\begin{equation*}
A_p(u_0,u_1)(t)\,:=\,D^pS(t)_{(u_0,u_1),\ldots, (u_0,u_1)}(0,0).
\end{equation*}
We work with data given by
\begin{equation}
\aligned
\widehat{u^N_0}(\xi)\,=\,\varphi_{B_N}(\xi)\,+\,\varphi_{-B_N}(\xi),\qquad
\widehat{u^N_1}(\xi)\,=\,-i \lambda(\xi)\,\left(\varphi_{B_N}(\xi)\,-\,\varphi_{-B_N}(\xi)\right),
\endaligned
\label{u0n-u1n}
\end{equation}
where $(B_N)_{N\geq 1}$ is a sequence of subsets of $\R$ and $\varphi_A$ is the characteristic function of the set $A$. It is easy to check that 
\begin{equation*}
\overline{\widehat{u^N_0}(\xi)}\,=\,\widehat{u^N_0}(-\xi),\qquad \overline{\widehat{u^N_1}(\xi)}\,=\,\widehat{u^N_1}(-\xi),
\end{equation*}
and, thus, our data are real-valued. By using \eqref{L} and \eqref{dps}, we infer 
\begin{equation*}
\widehat{L(u^N_0,u^N_1)(t)}(\xi)\,=\, e^{-it\lambda(\xi)}\varphi_{B_N}(\xi)\,+\,e^{it\lambda(\xi)}\varphi_{-B_N}(\xi)
\label{l+}
\end{equation*}
and, subsequently,
\begin{equation}
\aligned
\widehat{A_p(u^N_0,u^N_1)(t)}&(\xi)\\ 
= \mp p!\,\lambda(\xi)&\int_0^t \,\sin((t-\tau) \lambda(\xi))
\left\{\int\limits_{\eta_1+\ldots+\eta_p=\xi}
\prod_{j=1}^{p}\varphi_{\pm B_N}(\eta_j)\cdot e^{\mp i\tau\lambda(\eta_j)}\right\}d\tau,
\endaligned
\label{apn}
\end{equation}
where the inner integral is defined according to
\[
\int\limits_{\eta_1+\ldots+\eta_p=\xi}f:=\int_{\R^{p-1}}f(\eta_1,\ldots, \eta_{p-1}, \xi-\eta_1-\ldots-\eta_{p-1})\,d\eta_1\ldots d\eta_{p-1}
\]
and, with its integrand, we assumed an Einstein summation convention for the symbol $\pm$; i.e., if $\eta\in \pm \,B_N$, then the corresponding exponent is $\mp \,it\lambda(\eta)$.

As one notices, the generic term in the time integral which yields $\widehat{A_p(u^N_0,u^N_1)(t)}(\xi)$ is of the type 
\begin{equation*}
\int_0^t\,\sin(\alpha(t-\tau))\,e^{i\beta \tau}\,d\tau,
\end{equation*}
where
\begin{equation}
\aligned
&\alpha=\lambda(\xi), \qquad \beta= -\,\epsilon_1\lambda(a_1)-\epsilon_2\lambda(a_2)-\ldots-\epsilon_p\lambda(a_p),\\
\xi=\ & \epsilon_1 a_1+\epsilon_2 a_2+ \ldots+\epsilon_p a_p, \quad \epsilon_j=\pm \,1, \ a_j\in B_N, \ (\forall)1\leq j\leq p.
\endaligned
\label{ab}
\end{equation}
Moreover, a direct computation reveals that for real parameters $\alpha$ and $\beta$ we have
\begin{equation}
\text{Re}\left\{\int_0^t\,\sin(\alpha(t-\tau))\,e^{i\beta \tau}\,d\tau\right\}=
\begin{cases}
\frac{\alpha}{\alpha^2-\beta^2}(\cos (\beta t)-\cos (\alpha t)), \ &\text{for}\ |\alpha|\neq |\beta|,\\
\\
\frac{1}{2}t\,\sin(\alpha t), \ &\text{for}\ |\alpha|= |\beta|.
\end{cases}
\label{re-int}
\end{equation}

There are two key facts which allow us to argue for \eqref{ratio}. The first one is the localization in frequency of our data, which is enforced by choosing
\begin{equation*}
B_N=[N,N+1], \qquad (\forall)\,N\geq 1.
\end{equation*}
Coupled with \eqref{u0n-u1n}, this localization easily implies
\begin{equation}
\left\|u_0^N\right\|_{H^{s}} + \left\|u_1^N\right\|_ {H^{s}}\,\simeq\,N^{s}.
\label{data-hs}
\end{equation}
The second important point is that we are interested only in the output of $A_p(u^N_0,u^N_1)(t)$ at preferred frequencies, depending on the parity of $p$. This enables us to have control on the relative size of the parameter $\beta$ in \eqref{ab}, which in turn reduces the argument to obtaining good asymptotics for the generic term. 


\subsection{Argument for $p$ even}
In this case, we restrict our attention to the behavior of $\widehat{A_p(u^N_0,u^N_1)(t)}$ on the interval $[1/4,1/2]$ and first deduce that
\begin{equation}
\aligned
\left\|A_{p}\left(u_0^N,u_1^N\right)(t)\right\|_{H^s}\,&\geq\, \left\|A_{p}\left(u_0^N,u_1^N\right)(t)\right\|_{H^{s}([\frac{1}{4},\frac{1}{2}])}\\
&\simeq\, \left\|A_{p}\left(u_0^N,u_1^N\right)(t)\right\|_{L^2([\frac{1}{4},\frac{1}{2}])}.
\endaligned
\label{aphs}
\end{equation}
Next, due to \eqref{ab}, we obtain that for $N$ sufficiently large (depending on $p$) we must have an equal number of $+1$s and $-1$s in \eqref{ab} for $\xi\in  [1/4,1/2]$ to be true. Thus, eventually relabelling the indices, we can write 
\begin{equation*}
\xi=a_1-a_2+\ldots+a_{p-1}-a_p, \qquad \beta=\l(a_1)-\l(a_2)+\ldots+\l(a_{p-1})-\l(a_p).
\end{equation*}
 
Following this, we use \eqref{la-lb} to infer that
\begin{equation}
\left|\lambda(a)-\lambda(b)\right|\,\lesssim \,\frac{1}{N^3}, \qquad (\forall)\, a, b \in B_N, 
\label{la-lb-n}
\end{equation}
which leads to $|\beta|\lesssim 1/N^3$. We also notice that $\alpha=\l(\xi)\simeq 1$ if $\xi\in  [1/4,1/2]$. On the basis of \eqref{re-int}, we derive that for such values of $\alpha$ and $\beta$,
\begin{equation*}
\text{Re}\left\{\int_0^t\,\sin(\alpha(t-\tau))\,e^{i\beta \tau}\,d\tau\right\}\simeq \sin^2(\alpha t)+\text{O}\left(\frac{1}{N^6}\right)
\end{equation*}
holds true if $0<t<1$ and $N$ is large enough.

These facts tell us that, for fixed $0<t<1$, the real part of $\widehat{A_p(u^N_0,u^N_1)(t)}$ is correctly described by the real part of the generic term in \eqref{apn}. As a consequence, we obtain 
\[
\liminf_{N\to\infty} \left\|A_{p}\left(u_0^N,u_1^N\right)(t)\right\|_{L^2([\frac{1}{4},\frac{1}{2}])} \gtrsim\left(\int_\frac{1}{4}^\frac{1}{2}\,\sin^4(\lambda(\xi)t)\,d\xi\right)^\frac{1}{2}
\]
and, factoring in \eqref{data-hs} and \eqref{aphs}, we conclude that
\begin{equation*}
\lim_{N\to \infty}\, \frac{\left\|A_p\left(u^N_0,u^N_1\right)(t)\right\|_{H^s}}{\left(\left\|u_0^N\right\|_{H^{s}} + \left\|u_1^N\right\|_ {H^{{s}}}\right)^p}\,=\,\infty, \qquad (\forall)\,0<t<1.
\end{equation*}
This proves \eqref{ratio} in the case when $p$ is even.


\subsection{Argument for $p$ odd}
For this scenario, we focus on how $\widehat{A_p(u^N_0,u^N_1)(t)}$ evolves on the interval $[N,N+1]$ and, accordingly, proceed with
\begin{equation}
\aligned
\left\|A_{p}\left(u_0^N,u_1^N\right)(t)\right\|_{H^s}\,&\geq\, \left\|A_{p}\left(u_0^N,u_1^N\right)(t)\right\|_{H^{s}([N,N+1])}\\
&\simeq\, N^s\left\|A_{p}\left(u_0^N,u_1^N\right)(t)\right\|_{L^2([N,N+1])}.
\endaligned
\label{aphs-v2}
\end{equation}
Arguing as in the even case, we deduce that the representation of $\xi\in [N,N+1]$ in \eqref{ab} requires precisely one more $+1$ than $-1$s. Hence, we obtain 
\begin{equation*}
\xi=a_1-a_2+\ldots-a_{p-1}+a_p, \qquad \beta=\l(a_1)-\l(a_2)+\ldots-\l(a_{p-1})+\l(a_p),
\end{equation*}
following a possible relabelling of the indices.

Next, it is straightforward to derive
\begin{equation*}
\left|\lambda(a)-1\right|\,\simeq \,\frac{1}{N^2}, \qquad (\forall)\, a\in B_N.
\end{equation*}
Jointly with \eqref{la-lb-n}, this estimate implies
\begin{equation*}
\alpha=1+\text{O}\left(\frac{1}{N^2}\right),\qquad -\beta=1+\text{O}\left(\frac{1}{N^2}\right).
\end{equation*}
Invoking \eqref{re-int} again, we infer that
\begin{equation*}
\text{Re}\left\{\int_0^t\,\sin(\alpha(t-\tau))\,e^{i\beta \tau}\,d\tau\right\}\simeq t\sin(t)+\text{O}\left(\frac{1}{N^2}\right)
\end{equation*}
is valid for $0<t<1$ and $N$ big enough.

As in the case when $p$ is even, the real part of the generic term in \eqref{apn} describes accurately the real part of $\widehat{A_p(u^N_0,u^N_1)(t)}$ and, thus,
\[
\left\|A_{p}\left(u_0^N,u_1^N\right)(t)\right\|_{L^2([N,N+1])} \simeq t\sin(t)+\text{O}\left(\frac{1}{N^2}\right).
\]
Now, we can use \eqref{data-hs} and \eqref{aphs-v2} to conclude that
\[
\frac{\left\|A_p\left(u^N_0,u^N_1\right)(t)\right\|_{H^s}}{\left(\left\|u_0^N\right\|_{H^{s}} + \left\|u_1^N\right\|_ {H^{{s}}}\right)^p}\,\simeq\,\frac{t\sin(t)+\text{O}(1/N^2)}{N^{s(p-1)}}, \qquad (\forall)\,0<t<1,
\]
which yields \eqref{ratio} also in the odd case.

\section*{Acknowledgements}
The first author was supported in part by a grant from the Simons Foundation $\#\, 359727$. Both authors are grateful to Allan Greenleaf and Alex Iosevich for helpful discussions during the preparation of this manuscript.

\bibliographystyle{amsplain}
\bibliography{bousbib}

\end{document}